\documentclass[12pt]{amsart}
\usepackage{amsmath,amsthm,amsfonts,amssymb,latexsym,enumerate} 
\usepackage{showlabels}
\usepackage[pagebackref]{hyperref} 

\newcommand{\bC}{{\mathbf C}}

\newcommand{\bF}{{\mathbf F}}

\newcommand{\bO}{{\mathbf O}}

\newcommand{\SSS}{\mathsf{S}}

\newcommand{\CCC}{\mathsf{C}}

\newcommand{\Aut}{{{\operatorname{Aut}}}}

\newcommand{\Irr}{{{\operatorname{Irr}}}}

\newcommand{\cd}{\operatorname{cd}}
\newcommand{\dl}{\operatorname{dl}}

\newcommand{\cod}{\operatorname{cod}}

\newcommand{\Ker}{\operatorname{Ker}}

\newtheorem{thm}{Theorem}[section]
\newtheorem{lem}[thm]{Lemma}

\newtheorem*{thmA}{Theorem A}
\newtheorem*{conA'}{Conjecture A'}
\newtheorem*{thmB}{Theorem B}
\newtheorem*{thmC}{Theorem C}
\newtheorem*{thmD}{Theorem D}

\theoremstyle{definition}

\numberwithin{equation}{section}

\begin{document}

\title[Character codegrees]{A dual version of Huppert's $\rho$-$\sigma$ conjecture for character codegrees}

\author{Alexander Moret\'o}
\address{Departamento de Matem\'aticas, Universidad de Valencia, 46100
  Burjassot, Valencia, Spain}
\email{alexander.moreto@uv.es}

\thanks{Research  supported by Ministerio de Ciencia e Innovaci\'on PID-2019-103854GB-100, FEDER funds  and Generalitat Valenciana AICO/2020/298.}



\date{\today}

\begin{abstract}
We classify the finite groups with the property that any two different character codegrees are coprime. In general, 
we conjecture that if $k$  is a positive integer such that for any prime $p$ the number of character codegrees of a finite group $G$ that are divisible by $p$ is at most $k$, then the number of prime divisors of $|G|$ is bounded in terms of $k$. We prove this conjecture for solvable groups.
\end{abstract}

\maketitle


\section{Introduction}  

One of the main open problems on character degrees of finite groups is Huppert's $\rho$-$\sigma$ conjecture (see \cite{adp} for the  most recent progress on this problem and for its history). A dual problem was considered in \cite{ben,mm}, for instance.  Suppose that $G$ is a finite group. Let $\cd(G)$ be the set of degrees of the complex irreducible characters of $G$. Suppose that $k$ is an integer such that for every prime $p$, $p$ divides at most $k$ members of $\cd(G)$.  It was proved in \cite{ben} that if $G$ is solvable then $|\cd(G)|$ is bounded by a quadratic function of $k$. The existence of a bound for arbitrary finite groups was shown in \cite{mm}.
As pointed out by D. Gluck in his MathSciNet review of \cite{mm}, this formally dual problem is perhaps more difficult.  In fact, it is still unknown whether a linear bound exists, even in the solvable case.

There has been a lot of recent interest in character codegrees. If we write $\Irr(G)$ to denote the set of complex irreducible characters of $G$ and $\chi\in\Irr(G)$ then the codegree of $\chi$ is 
$$
\cod(\chi)=\frac{|G:\Ker\chi|}{\chi(1)}.
$$
We write $\cod(G)=\{\cod(\chi)\mid\chi\in\Irr(G)\}$ to denote the set of character codegrees.  This concept was introduced in \cite{qww}.  The analogue of Huppert's $\rho$-$\sigma$ conjecture for character codegrees has already been studied in several papers (see, for example, \cite{yq, mor21}). In this note, we consider the dual problem for character codegrees and obtain an almost quadratic bound.

\begin{thmA}
Let $k$ be a positive integer. If $G$ is a finite solvable group with the property that for any prime $p$, $p$ divides at most $k$ members of $\cod(G)$, then 
$$
|\cod(G)|\leq 24k^2\log_2k+390k^2+1.
$$
\end{thmA}

We obtain Theorem A as a consequence of the following result. Given a group $G$, $\pi(G)$ stands for the set of prime divisors of $|G|$.

\begin{thmB}
Let $k$ be a positive integer. If $G$ is a finite solvable group with the property that for any prime $p$, $p$ divides at most $k$ members of $\cod(G)$, then 
$$
|\pi(G)|\leq 24k\log_2k+390k. 
$$
\end{thmB}

In turn, the following result is fundamental in our proof of Theorem B. Given a finite solvable group $G$, $\dl(G)$ stands for the derived length of $G$, $$\cod_p(G)=\{n\in\cod(G)\mid\, p\mid n\}$$ is the set of character codegrees of $G$ that are divisible by $p$ and if $p$ is a prime then $\bO_p(G)$ and $\bO_{p'}(G)$ are the largest normal $p$-subgroup of $G$ and the largest normal $p'$-subgroup of $G$, respectively.

\begin{thmC}
Let $p$ be a prime. If $G$ is a finite solvable group with $\bO_{p'}(G)=1$ then 
$$
\dl(G/\bO_p(G))\leq24\log_2|\cod_p(G)|+389.
$$
\end{thmC}

It has been proven in Theorem 1.5 of \cite{aha} that
 the $p$-length of a $p$-solvable group is at most the number of character codegrees that are divisible by $p$.  Theorem C is, at least asymptotically, an improvement of this result for solvable group.
 
 We prove Theorems A, B and C in Section 2. Our proofs are short, but they rely on a deep theorem of T. M. Keller on orbit sizes \cite{kel} that had not been used before in this context.
 These theorems lead us to several open questions that will be discussed in Section 3. While we have not obtained any general results for nonsolvable groups, we classify the groups (solvable or not) with the property that for any prime $p$, $p$ divides at most one character codegree.
 
 \begin{thmD}
 Let $G$ be a finite group. Then any two different character codegrees are coprime if and only if one of the following holds:
 \begin{enumerate}
 \item
 $G$ is an elementary abelian $p$-group for some prime $p$. In this case, $\cod(G)=\{1,p\}$.
 \item
 $G$ is a Frobenius group with complement of prime order $p$, kernel a $q$-group for some prime $q$ and $\cod(G)=\{1,p,q^s\}$ for some integer $s\geq 1$.
 \end{enumerate}
In particular,  if any two different character codegrees are coprime then $G$ is solvable, $|\pi(G)|\leq2$ and  $|\cod(G)|\leq3$. 
 \end{thmD}

 \section{Proof of Theorems A, B and C}
 
 We will use without further explicit mention the following properties of character codegrees. This is \cite[Lemma 2.1]{qww}.
 
 \begin{lem}
 Let $G$ be a finite group and let $N\trianglelefteq G$. Then
 \begin{enumerate}
 \item
 Let $\chi\in\Irr(G)$ such that $N\leq\Ker\chi$. The codegree of $\chi$ viewed as a character of $G/N$ and the codegree of $\chi$ viewed as a character of  $G$ coincide.
  \item
 If $\theta\in\Irr(N)$ and $\chi\in\Irr(G)$ lies over $\theta$, then $\cod(\theta)$ divides $\cod(\chi)$.
 \end{enumerate}
 \end{lem}
 
 The following deep theorem \cite[Theorem 2.1]{kel} is the key to our proofs.
 
 \begin{thm}
 \label{orb}
 Let $G$ be a solvable group acting faithfully and irreducibly on a vector space $V$ over a finite field. Then
 $$
 \dl(G)\leq24\log_2 m^*(G,V)+389,
 $$
 where $m^*(G,V)$ is the number of orbit sizes of $G$ on $V$ that exceed $1$.
 \end{thm}
 
 \begin{proof}
 This is essentially \cite[Theorem 2.1]{kel}. The only difference is that we are considering just non-trivial orbits and we have used the crude bound $\log_2m(G,V)\leq\log_2m^*(G,V)+1$.
 \end{proof}

  Now, we are ready to prove Theorem C. If $N\trianglelefteq G$ we write $$\Irr(G|N)=\{\chi\in\Irr(G)\mid\chi\not\leq\Ker\chi\}.$$ We also write $\cd(G|N)$ and $\cod(G|N)$ to denote the sets of degrees and codegrees, respectively, of the characters in $\Irr(G|N)$.
   
  \begin{proof}[Proof of Theorem C]
  We may assume that $\Phi(G)=1$.  By G\"aschutz's theorem \cite[Theorem 1.12]{mw}, $G$ splits over $\bF(G)=\bO_p(G)$.  Hence, we can decompose $G$ as the semidirect product $G=HV$,
  where  $V=\bO_p(G)$ is a completely reducible faithful $H$-module over the field with $p$ elements. Write $V=V_1\oplus\cdots\oplus V_t$ as a direct sum of irreducible $H$-modules $V_i$. 
  Note that $H\cong G/\bO_p(G)$ and $H/\bC_H(V_i)$ acts faithfully and irreducibly on $V_i$ for every $i=1,\dots,n$. Note that $H$ embeds into $H/\bC_H(V_i)\times\cdots\times H/\bC_H(V_t)$, so 
  $$
  \dl( G/\bO_p(G))=\dl(H)\leq\max\{\dl(H/\bC_H(V_1)),\dots,\dl(H/\bC_H(V_t))\}.
  $$
  
  For every $i$, write $V=V_i\oplus W_i$, where $W_i$ is the direct sum of the remaining $V_i$'s. Note that  $G/W_i\cong HV_i$. Hence, the semidirect product $\Gamma$ of $ H/\bC_H(V_i)$ acting on $V_i$ is isomorphic to a quotient of $G$. In particular, $\cod(\Gamma)\subseteq\cod(G)$. 
  
  Note that for every $1_{V_i}\neq\lambda\in\Irr(V_i)$, $p$ divides $\cod(\lambda)$. Hence, $p$ divides $\cod(\chi)$ for every $\chi\in\Irr(\Gamma|V_i)$. On the other hand, since $V_i$ is the unique minimal normal subgroup of $\Gamma$, all the characters in $\Irr(G|V_i)$ are faithful. This implies that $$|\cd(\Gamma|V_i)|=|\cod(\Gamma|V_i)|.$$ Furthermore, by Problem 6.18 of \cite{isa}, for every  $1_{V_i}\neq\lambda\in\Irr(V_i)$, $\lambda$ extends to $I_{\Gamma}(\lambda)$. It follows from Clifford's correspondence \cite[Theorem 6.16]{isa} that $$m^*(H/\bC_H(V_i),\Irr(V_i))\subseteq\cd(\Gamma|V_i).$$ We conclude that 
  $$
  m^*(H/\bC_H(V_i),\Irr(V_i))\leq|\cod(\Gamma|V_i)|\leq|\cod_p(\Gamma)|\leq|\cod_p(G)|.
  $$
  Notice that since $H/\bC_H(V_i)$ acts faithfully and irreducibly on $V_i$, it also acts faithfully and irreducibly on the dual group $\Irr(V_i)$ by \cite[Proposition 12.1]{mw}.
   Hence, by Theorem \ref{orb},
  $$
  \dl(H/\bC_H(V_i))\leq24\log_2 m^*(H/\bC_H(V_i),\Irr(V_i))+389\leq24\log_2|\cod_p(G)|+389.
$$
  Therefore,
  $$
   \dl(G/\bO_p(G))\leq\max\{\dl(H/\bC_H(V_1)),\dots,\dl(H/\bC_H(V_t))\}\leq24\log_2|\cod_p(G)|+389,
   $$
   as wanted.
   \end{proof}

As discussed in \cite{mor1}, it is not clear whether the derived length of a $p$-group is bounded in terms of the number of character codegrees,  so it does not seem easy to remove $\bO_p(G)$ in  the statement of Theorem C.
  
  Next, we work toward a proof of Theorem B. We need the following lemma.
  
  \begin{lem}
  \label{new}
  Let $G$ be a group. Assume that for every prime $p$, $G$ has at most $k$ character codegrees that are divisible by $p$. Let $L\leq K$ be normal subgroups of $G$ such that $K/L$ is nilpotent. Then $|\pi(K/L)-\pi(G/K)|\leq k$.
  \end{lem}
  
  \begin{proof}
  We have to show that the number of ``new" prime divisors that appear in $K/L$ is at most $k$. Let $$\pi(K/L)-\pi(G/K)=\{p_1,\dots,p_t\}=\pi.$$ Write $$K/L=P_1/L\times\cdots\times P_t/L\times H/L,$$ where $P_i/L$ is a Sylow $p_i$-subgroup of $K/L$ for every $i$ and $H/L$ is a Hall $\pi$-complement. Notice that the codegree of a linear character of order $p_1\cdots p_i$ is $p_1\cdots p_i$ for every $i$. Take  
  $\lambda_i\in\Irr(K/L)$ of order  $p_1\cdots p_i$ for every $i$. Notice that $P_{i+1}\cdots P_tH\leq\Ker\lambda_i$ for every $i<t$, whence $\lambda_i\in\Irr(K/P_{i+1}\cdots P_tH)$. We adopt the convention that if $i=t$ then $P_{i+1}\cdots P_tH=H$, so that the previous sentence also holds for $i=t$. 
  
  Let $\chi_i\in\Irr(G)$ lying over $\lambda_i$ for every $i$. Note that $\chi_i\in\Irr(G/P_{i+1}\cdots P_tH)$ has codegree a multiple of $p_1\cdots p_i$ by Lemma 2.1.  Notice that if $1\leq j<k\leq t$ then $p_k$ does not divide $|G/P_{j+1}\cdots P_tH|$ so the codegrees of $\chi_1,\dots,\chi_t$ are pairwise different. Thus, we have found $t$ different character codegrees that are multiples of $p_1$. Hence $t\leq k$, as wanted.
  \end{proof}

 \begin{proof}[Proof of Theorem B]
 Let $p$ be a prime divisor of $|G|$. Recall that $\bO_{p',p}(G)$ is the inverse image in $G$ of $\bO_p(G/\bO_{p'}(G))$. 
 Applying Theorem C to $G/\bO_{p'}(G)$ we deduce that 
 $$
 \dl(G/\bO_{p',p}(G))\leq 24\log_2|\cod_p(G)|+389\leq 24\log_2k+389.
 $$

Note that by elementary group theory the intersection of $\bO_{p',p}(G)$ when $p$ runs over the prime divisors of $|G|$ is contained in the Fitting subgroup  $\bF(G)$. Hence
 $$
 \dl(G/\bF(G))\leq\dl(G/\bigcap_{p\mid|G|}\bO_{p',p}(G))\leq\max_{p\mid|G|}\dl(G/\bO_{p',p}(G))\leq24\log_2k+389.
 $$

 This implies that there exist $1=N_0\leq N_1\leq\cdots\leq N_{l-1}\leq N_l=G$ normal subgroups of $G$ such that $l\leq24\log_2k+390$ and $N_{i+1}/N_i$ is nilpotent for every $i$ (in fact, we may assume that it is abelian for $i>0$). By the previous lemma, the number of ``new" prime divisors in every factor $N_{i+1}/N_i$ is at most $k$. The result follows.
 \end{proof}
 
 \begin{proof}[Proof of Theorem A]
 It suffices to apply Theorem B and the pigeonhole principle.
 \end{proof}
 
 \section{Proof of Theorem D}
 
 While Benjamin's theorem \cite{ben} only uses standard character theory, here we have needed to use  Keller's large orbit theorem. In Benjamin's paper, it was suggested that if $G$ is solvable and for every prime $p$ there exist at most $k$ character degrees that are divisible by $p$ then, perhaps, $|\cd(G)|\leq3k$.  We think that the right bound in Theorem A is linear, but it is not clear even  what this linear bound should be. For instance, it is not true that if for every prime $p$, $p$ divides at most $k$ character codegrees, then  $|\cod(G)|\leq3k$. Consider for instance the Frobenius group $G=\CCC_6\ltimes\CCC_{7\cdot13}$. It is easy to check that $\cod(G)=\{1,2,3,6,7,13,91\}$, so for this group every prime divides at most $2$ character codegrees while $|\cod(G)|=7$.
 
 We conjecture that Theorems A and B should hold for arbitrary finite groups (and with linear bounds). Using the classification of finite simple groups, we can obtain sharp bounds when $k=1$.  
 We need the following lemma.
 
  \begin{lem}
 \label{sim}
 Let $S$ be a nonabelian simple group.  Suppose that $S\not\in\{J_4, Fi'_{24}, M\}$. Then there exists $\theta\in\Irr(S)$ that extends to $\Aut(S)$ such that $\cod(\theta)$ is divisible by all the prime divisors of $|G|$ except for at most one. If $S\in\{J_4, Fi'_{24}, M\}$ then there exist $\theta_1,\theta_2\in\Irr(S)$ that extend to $\Aut(S)$ such that $\cod(\theta_1)\neq\cod(\theta_2)$ and $(\cod(\theta_1),\cod(\theta_2))>1$.
 \end{lem}
 
 \begin{proof}
 The first part is \cite[Lemma 2.1(i)]{mor21}. The second part can be checked in the Atlas \cite{atl}. (Note that since $S$ is simple, $\cod(\theta)=|S|/\theta(1)$ for any non-principal $\theta\in\Irr(S)$.)
 \end{proof}
 
 Now, we complete the proof of Theorem D, which we restate for the convenience of the reader.
 
 \begin{thm}
 Let $G$ be a finite group. Then any two different character codegrees are coprime if and only if one of the following holds:
 \begin{enumerate}
 \item
 $G$ is an elementary abelian $p$-group for some prime $p$. In this case, $\cod(G)=\{1,p\}$.
 \item
 $G$ is a Frobenius group with complement of prime order $p$, kernel a $q$-group for some prime $q$ and $\cod(G)=\{1,p,q^s\}$ for some integer $s\geq 1$.
 \end{enumerate}
 \end{thm}
 
 \begin{proof}
 
It is clear that if either (i) or (ii) holds, then any two different character codegrees are coprime. Assume now that any two different character codegrees are coprime. 
 First, we want to see that $G$ is solvable. Let $M$ be the solvable residual. By way of contradiction, suppose that $M>1$. Let $M/N=S\times\cdots\times S$, where $S$ is a nonabelian simple group,  be a chief factor of $G$. Let $C=\bC_{G/N}(M/N)$, so that $G/C$ has a unique minimal normal subgroup $K/C$ which is isomorphic to $M/N$. Therefore, $G$ is isomorphic to a subgroup of $\Gamma=\Aut(S)\wr\SSS_n$, where $n$ is the number of copies of $S$ that appear as direct factor of $M/N$.
 
 Suppose first that $S\not\in\{J_4, Fi'_{24}, M\}$. By Lemma \ref{sim}, there exists $\theta\in\Irr(S)$ that extends to $\Aut(S)$ such that $\cod(\theta)$ is divisible by all the prime divisors of $|S|$ except for at most one. Now, using Lemma 1.3 of \cite{mat}, we deduce that $\psi=\theta\times\cdots\times\theta\in\Irr(K/C)$ extends to $\Gamma$. 
Hence, it also extends to $G$. Let $\tilde{\psi}\in\Irr(G/C)$ be an  extension of $\psi$. Notice that $\tilde{\psi}$ is faithful (because $\psi$ is faithful and $K/C$ is the unique minimal normal subgroup of $G$. Hence, $$\cod(\tilde{\psi})=|G:K|\cod(\psi)$$ is divisible by all the prime divisors of $|G/C|$ except for at most one. Suppose that $p$ is such a prime. By hypothesis, $G/C$ has at most one character codegree that is divisible by $p$. Using our hypothesis again, we conclude that $$\cod(G/C)\subseteq\{1,p^a, \cod(\tilde{\psi})\}$$ has cardinality at most $3$.  By Theorem 3.4 of \cite{5}, this implies that $G/C$ is solvable, a contradiction.

Finally, assume that $S\in\{J_4, Fi'_{24}, M\}$.  By Lemma \ref{sim}, there exist $\theta_1,\theta_2\in\Irr(S)$ that extend to $\Aut(S)$ such that $\cod(\theta_1)\neq\cod(\theta_2)$ and $(\cod(\theta_1),\cod(\theta_2))>1$.  Using Lemma 1.3 of \cite{mat} again, we deduce that $\psi_1=\theta_1\times\cdots\times\theta_1$ and $\psi_2=\theta_2\times\cdots\times\theta_2\in\Irr(K/C)$  extend to $\Gamma$. In particular $\psi_1$ and $\psi_2$ extend to $\tilde{\psi_1},\tilde{\psi_2}\in\Irr(G/C)$, respectively. 
As before, $\cod(\tilde{\psi_1})=|G:K|\cod(\psi_1)$ and $\cod(\tilde{\psi_2})=|G:K|\cod(\psi_2)$. Therefore, $G/C$ has two different character codegrees that are not coprime. This contradiction proves that $G$ is solvable. 

Now, we want to prove that $|\cod(G)|\leq3$. By way of contradiction, suppose that $|\cod(G)|\geq 4$.  Let $\Gamma(G)$ be the graph whose vertices are the prime divisors of $|G|$ and two vertices $p$ and $q$ are joined by an edge if and only if $pq$ divides some character codegree of $G$. Note that our hypothesis implies that the number of connected components of this graph is  $|\cod(G)|-1\geq 3$. By the Corollary to Theorem E of \cite{qww}, this implies that the prime graph associated to element orders of $G$ has at least $3$ connected components. Since $G$ is solvable, this contradicts a theorem of Gruenberg and Kegel (see \cite{wil}).  Therefore, $|\cod(G)|\leq 3$.  Now, the theorem follows from Lemma 2.4 of \cite{dl} and Theorem 3.4 of \cite{5}. 
\end{proof}

\end{document}